\documentclass[11pt]{article}
\usepackage[margin=2.5cm]{geometry}
\usepackage{graphicx}

\usepackage[pagebackref, pdftex]{hyperref}

\renewcommand*{\backref}[1]{}
\renewcommand*{\backrefalt}[4]{
  \ifcase #1
  [No citations.]
  \or [#2]
  \else [#2]
  \fi }

\usepackage{amsmath}
\usepackage{amssymb}
\usepackage{amsthm}
\usepackage{cite}
\usepackage{tikz}
\usepackage{pgfplots}
\usepackage{bbm}
\usepackage{authblk}

\makeatletter
\renewcommand{\qed}{$\hfill\square$}

\numberwithin{equation}{section}

\newtheorem{theorem}[equation]{Theorem}

\newtheorem{lemma}[equation]{Lemma}
\newtheorem{lem}[equation]{Lemma}

\newtheorem{prop}[equation]{Proposition}

\newtheorem{definition}[equation]{Definition}

\newtheoremstyle{named}{}{}{\itshape}{}{\bfseries}{.}{.5em}{\thmnote{#3} #1}
\theoremstyle{named}
\newtheorem*{namedtheorem}{Theorem}

\newcommand{\refsec}[1]{Section~\ref{Sec:#1}}
\newcommand{\refdef}[1]{Definition~\ref{Def:#1}}
\newcommand{\reffig}[1]{Figure~\ref{Fig:#1}}

\newcommand{\refeqn}[1]{\eqref{Eqn:#1}}
\newcommand{\reflem}[1]{Lemma~\ref{Lem:#1}}
\newcommand{\refprop}[1]{Proposition~\ref{Prop:#1}}
\newcommand{\refthm}[1]{Theorem~\ref{Thm:#1}}

\DeclareMathOperator{\Isom}{Isom}
\let\Im\relax
\DeclareMathOperator{\Im}{Im}
\let\Re\relax
\DeclareMathOperator{\Re}{Re}
\newcommand{\R}{\mathbb{R}}
\newcommand{\C}{\mathbb{C}}
\newcommand{\Z}{\mathbb{Z}}
\newcommand{\hyp}{\mathbb{H}}
\newcommand{\D}{\mathbb{D}}
\newcommand{\U}{\mathbb{U}}
\newcommand{\To}{\longrightarrow}

\begin{document}

\title{Spinors and Descartes' Theorem}

\author{Daniel V. Mathews\thanks{Daniel.Mathews@monash.edu} }

\author{Orion Zymaris\thanks{Orion.Zymaris@monash.edu}}

\affil{School of Mathematics,
Monash University,
VIC 3800, Australia}

\maketitle

\begin{abstract}
Descartes' circle theorem relates the curvatures of four mutually externally tangent circles, three ``petal" circles around the exterior of a central circle, forming a ``$3$-flower" configuration. We generalise this theorem to the case of an ``$n$-flower", consisting of $n$ tangent circles around the exterior of a central circle, and give an explicit equation satisfied by their curvatures. The proof uses a spinorial description of horospheres in hyperbolic geometry.
\end{abstract}
\section{Introduction}
\emph{Descartes' theorem} is a classical result in 2-dimensional Euclidean geometry, relating the curvatures of four mutually tangent circles (\reffig{3Flower}) which form a \emph{3-flower} in the following sense.
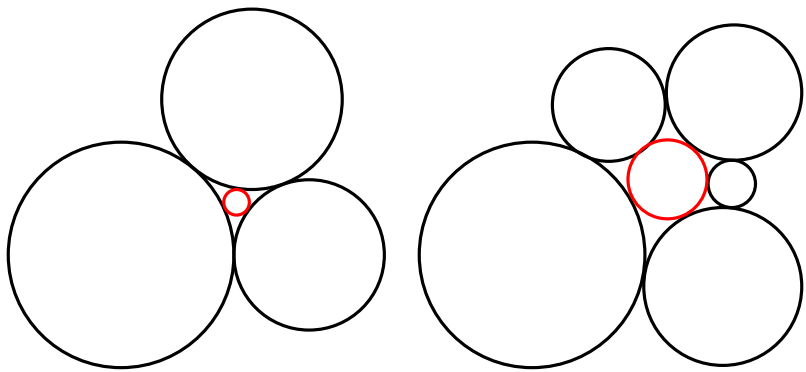
\begin{figure}[!h]
\centering
\begin{tabular}{cc}
\begin{tikzpicture}[scale=1]
\draw[very thick](0,0) circle (1.5);
\draw[very thick](2.5,0) circle (1);
\draw[very thick](1.74,2.07) circle (1.2);
\draw[very thick, red] (1.535,0.7) circle (0.17);
\end{tikzpicture}
&
\begin{tikzpicture}[scale=1.5]
\draw[very thick](0.47,0.01) circle (1);
\draw[very thick](2.16,-0.27) circle (0.7);
\draw[very thick](2.26,1.45) circle (0.6);
\draw[very thick](1.15,1.34) circle (0.5);
\draw[very thick](2.24,0.64) circle (0.21);
\draw[very thick, red] (1.67,0.68) circle (0.35);
\end{tikzpicture}
\end{tabular}
\caption{Left, a $3$-flower. Right, a $5$-flower.}
\label{Fig:3Flower}
\end{figure}
\begin{definition}
\label{Def:flower}
Let $n \geq 3$. An \emph{$n$-flower} consists of a central circle $C_\infty$, and $n$ \emph{petal} circles $C_j$, over integers $j$ mod $n$, so that the $C_j$ are externally tangent to $C_\infty$ in order around $C_\infty$, and each $C_j$ is externally tangent to $C_{j-1}$ and $C_{j+1}$.
\end{definition}
Throughout, we denote the curvature of a circle $C_\bullet$ by $\kappa_\bullet$. Descartes' theorem gives an equation satisfied by the curvatures in a 3-flower:
\begin{equation}
\label{Eqn:Descartes}
(\kappa_{\infty}+\kappa_1+\kappa_2+\kappa_3)^2=2(\kappa_{\infty}^2+\kappa_1^2+\kappa_2^2+\kappa_3^2).
\end{equation}
In this paper we present an explicit equation satisfied by the curvatures of an $n$-flower. 
\begin{namedtheorem}[Generalised Descartes]
\label{Thm:GeneralisedDescartes}
Let the circles $C_\infty$ and $C_j$ ($j \in \Z/n\Z$)  in an $n$-flower have curvatures $\kappa_\infty, \kappa_j$ respectively. Define $m_0$ and $m_j$ for $1 \leq j \leq n-1$ as
\begin{equation}
    \label{Eqn:mDefn}
m_0 = \sqrt{\frac{\kappa_0}{\kappa_{\infty}}+1}, \quad
m_j = \sqrt{\left(\frac{\kappa_j}{\kappa_{\infty}}+1\right)\left(\frac{\kappa_{j-1}}{\kappa_{\infty}}+1\right)-1}.
\end{equation}
Then for odd $n$, the following holds:
\begin{equation}
\label{Eqn:mquationOdd}
\frac{m_0^2 \, i}{2}\left(\prod_{j=1}^{n-1}(m_j-i) - \prod_{j=1}^{n-1}(m_j+i) \right)-\prod_{j=1}^{\frac{n-1}{2}} \left( m_{2j-1}^2+1 \right)=0.
\end{equation}
For even $n$, the following holds:
\begin{equation}
\label{Eqn:mquationEven}
\frac{i}{2}\left(\prod_{j=1}^{n-1}(m_j-i) - \prod_{j=1}^{n-1}(m_j+i) \right)-\prod_{j=1}^{\frac{n-2}{2}} \left( m_{2j}^2+1 \right) =0.
\end{equation}
\end{namedtheorem}
In other words, from the curvatures $\kappa_\bullet$ we define auxiliary variables $m_\bullet$ via \refeqn{mDefn}, and the $m_\bullet$ satisfy polynomial equations. Since all $\kappa_\bullet > 0$, each $m_\bullet$ is the square root of a manifestly positive number and we take $m_\bullet$ to be the positive square root. 

In \refeqn{mquationOdd} and \refeqn{mquationEven}, $i$ is the usual  square root of $-1$. Although complex numbers are crucial to the proof, in writing these equations they are merely a convenience. After expanding and cancelling terms, the resulting polynomials have integer coefficients. Indeed, writing $[n]$ for $\{1, 2, \ldots, n\}$, then for any subset $K \subseteq [n-1]$, the term in the product $\prod_{j=1}^{n-1} (m_j \pm i)$ involving precisely the $m_k$ with  $k \in K$ is  given by $(\pm i)^{n-1-|K|} \prod_{k \in K} m_k$. Such terms come in pairs, one from $\prod_{j=1}^{n-1} (m_j + i)$ and one from $\prod_{j=1}^{n-1} (m_j - i)$. When $n-1-|K|$ is even, the terms in each pair are real, equal and  cancel; when $n-1-|K|$ is odd, say equal to $2l+1$, then the terms in each pair are imaginary and conjugate, so since  $\frac{i}{2} \left( (-i)^{n-1-|K|} - i^{n-1-|K|} \right) = (-1)^l$, 
\refeqn{mquationOdd} and \refeqn{mquationEven} can be written  as
\begin{equation}
\label{Eqn:mquationOdd_alt}
m_0^2  \sum_{\stackrel{K \subseteq [n-1]}{|K| = n-2l-2}}  (-1)^l \prod_{k \in K}  m_k = \prod_{j=1}^{\frac{n-1}{2}} \left( m_{2j-1}^2 + 1 \right) 
\end{equation}
and
\begin{equation}
\label{Eqn:mquationEven_alt}
\sum_{\stackrel{K \subseteq [n-1]}{|K|=n-2l-2}} (-1)^l \prod_{k \in K} m_k = \prod_{j=1}^{\frac{n-2}{2}} \left( m_{2j}^2+1 \right)
\end{equation}
respectively.

In any case, making the substitutions of  \refeqn{mDefn} in \refeqn{mquationOdd} or \refeqn{mquationEven} (or \refeqn{mquationOdd_alt} or \refeqn{mquationEven_alt}) yields an equation satisfied by the $\kappa_\bullet$, providing a generalisation of Descartes' equation \refeqn{Descartes}. This equation involves square roots, but upon multiplying by various conjugates (replacing various $m_\bullet$  with  $-m_\bullet$) one may obtain a polynomial relation among the $m_\bullet^2$; after substituting and clearing denominators one obtains a polynomial relation among the $\kappa_\bullet$.

In this way, from equation \refeqn{mquationOdd} with $n=3$ one can recover Descartes' theorem. Similarly, from  \refeqn{mquationEven} with $n=4$ we obtain the following equation relating the curvatures in a 4-flower:
\begin{align*}
16 \kappa_\infty^4 &- 8 \kappa_\infty^2 (\kappa_1 \kappa_2 + \kappa_2 \kappa_3 + \kappa_3 \kappa_4 + \kappa_4 \kappa_1 + 2 \kappa_1 \kappa_3 + 2 \kappa_2 \kappa_4) \\
& + (\kappa_1^2 + \kappa_3^2)(\kappa_2^2 + \kappa_4^2) - 16 \kappa_\infty ( \kappa_1 \kappa_2 \kappa_3 + \kappa_2 \kappa_3 \kappa_4 + \kappa_3 \kappa_4 \kappa_1 + \kappa_4 \kappa_1 \kappa_2) \\
& - 12 \kappa_1 \kappa_2 \kappa_3 \kappa_4
- 2 (\kappa_1 \kappa_2 + \kappa_3 \kappa_4)(\kappa_2 \kappa_3 + \kappa_4 \kappa_1) = 0.
\end{align*}
For larger $n$ the polynomials grow rapidly in size and degree.

Flowers are a building block of circle packing theory \cite{StephensonKenneth2005Itcp}. It is known from the general theory of circle packing that once the curvatures of the petals of an $n$-flower are known, the curvature of the central circle is determined. \refthm{GeneralisedDescartes} gives an explicit equation for that central curvature.

\medskip

{\noindent \textbf{Overview of proof.} } 
The proof proceeds in two steps. First, we convert the problem to one in hyperbolic geometry. Inverting each petal $C_j$ of  an $n$-flower in the  central circle $C_\infty$ yields a collection of circles $\mathring{C}_j$ internally tangent to $C_\infty$: an  ``inverted flower". 
Viewing $C_\infty$ as the boundary of the disc model of the hyperbolic plane, each  $\mathring{C}_j$ appears as a horosphere. It is useful for our purposes to convert to the upper half plane model, where $C_\infty$ becomes $\R \cup \{\infty\}$, and each $\mathring{C}_j$ becomes a horosphere $\overline{C}_j$ appearing as a circle tangent to $\R \cup \{\infty\}$. 
It is not difficult to relate the Euclidean curvatures $\kappa_j, \mathring{\kappa}_j, \overline{\kappa}_j$ of the circles $C_j, \mathring{C}_j, \overline{C}_j$ to each other.

Second, we use the spinor-horosphere correspondence recently proved by the first author in \cite{Mathews_Spinors_horospheres}. Building on work of Penrose--Rindler \cite{PenroseRindler} and Penner \cite{PennerPunctured}, this correspondence provides an explicit, smooth, bijective, $SL(2,\C)$-equivariant correspondence between nonzero 2-component spinors, and horospheres in hyperbolic 3-space $\hyp^3$, with a certain spinorial decoration. For us, spinors can be regarded simply as pairs of complex numbers $(\xi, \eta)$. When $\xi,\eta$ are \emph{real}, the correspondence essentially reduces to 2 dimensions, yielding a correspondence between nonzero $(\xi, \eta) \in \R^2$, and horocycles in the hyperbolic plane $\hyp^2$, which is particularly simple in the upper half plane model.
Under the correspondence, there is a simple way to calculate the distance between horospheres, and hence to know when they are tangent, using the bilinear form on spinors given by the $2 \times 2$ determinant. Moreover, the curvatures $\mathring{\kappa}_j$ and $\kappa_j$ are both straightforward quadratic functions of $\xi, \eta$: they are each given, up to a constant, by the Euclidean square length $\xi^2 + \eta^2$ of $(\xi, \eta)$.

Thus, the problem reduces to finding a relation among the Euclidean lengths of real 2-dimensional vectors $(\xi, \eta)$, given that they satisfy certain bilinear conditions equivalent to the tangency relations of a flower. The bilinear conditions are essentially that the vectors successively span parallelograms of area $1$.
The desired relation is found by introducing complex variables $z_j = \xi_j + i \eta_j$.

The expressions for the $m_j$ in \refeqn{mDefn} arise naturally in calculations with the spinors $(\xi_j, \eta_j)$, as do the various products arising in the equations \refeqn{mquationOdd} and \refeqn{mquationEven}. We thus include some further calculations in \refsec{spinor_calc} which help to motivate and explain these expressions, and explain how \refeqn{mquationOdd} and \refeqn{mquationEven} were found.

\medskip

{\noindent \textbf{Related work.} } 
Descartes' equation dates back to 1643 and we do not attempt to provide a summary of the history or developments over subsequent centuries; we simply mention some notable and recent works, and those most closely related to our approach.

Recent historical accounts of the the correspondence between Descartes and Elisabeth of the Palatinate, in which equation \refeqn{Descartes} first appeared, include \cite{Bos2010}, \cite[pp. 31--33]{Nye_Princess}, and \cite[pp. 37--38, 73--81]{Shapiro_correspondence}.
Equation \refeqn{Descartes} was stated by Yamaji Nushizumi in 1751 \cite{Michiwaki_Japanese_geometry}, and rediscovered several times, including by Steiner (1826) \cite{Steiner_1826}, Beecroft (1842) \cite{Beecroft_1842}, and Soddy (1936) \cite{Soddy_1936}; the latter restating \refeqn{Descartes} in poetry. 

Numerous generalisations of Descartes' theorem are known; we mention a few. For instance, in 1936 Gosset generalised the result to $n+2$ mutually tangent spheres in $n$ dimensions, appending a verse to Soddy's poem \cite{Soddy_Gosset_nature}. In 1962 Mauldon generalised these results to spherical and hyperbolic space \cite{Mauldon_1962}. In 2002 Lagarias--Mallows--Wilks  \cite{LagariasMallowsWilks} further extended these results, relating not only the curvatures but also the centres of the spheres involved. In 2007, Kocik \cite{kocik2007theorem} extended the result for $n+2$ Euclidean spheres in $n$ dimensions to the case where circles need not be tangent.

Apollonian circle packings, consisting of nested 3-flowers, are a research field in their own right; see e.g. Graham--Lagarias--Mallows--Wilks--Yan \cite{GLMWY_2003} or Aharonov--Stephenson \cite{StephensonAharonov} for general background. These packings have important number-theoretic properties and have seen recent breakthroughs, e.g. \cite{HKRS_local-global}. 
A special case closely related to our work, and mentioned by the first author in \cite{Mathews_Spinors_horospheres}, is that of Ford circles \cite{Ford_1938}: these arise from \emph{integer} spinors. 

So far as we know, however, none of these works provide a generalisation of Descartes' theorem to $n$-flowers for $n > 3$. There are certainly results involving flower-like configurations, such as Soddy's hexlet \cite{Soddy_Hexlet}, but they do not provide an equation relating curvatures.

Flowers being simple examples of circle packings, the general theory of packing applies to them: see generally Stephenson \cite{StephensonKenneth2005Itcp}.
In general, from a simplicial 2-complex $K$  triangulating an oriented surface, circle packing theory studies the existence and uniqueness of circle packings realising $K$, in the sense that vertices of $K$ correspond to circles,  edges of $K$ correspond to tangencies, and triangles of $K$ correspond to oriented tangent circle triples. Flowers arise in the simple case when $K$ is a disc built from $n$ triangles around a central vertex.
Perhaps most relevantly for our purposes, the ``boundary value theorem" of \cite[thm. 11.6]{StephensonKenneth2005Itcp} provides that when $K$ is topologically a disc, the curvatures of circles at boundary vertices, and branching structure, may be specified arbitrarily, and then a unique packing (in Eudlicean or hyperbolic geometry) exists. Our result gives the Euclidean curvature of the interior circle in the flower case.

To the best of our knowledge, perhaps the existing works most closely related to our approach are those of Kocik, who in several papers uses spinors to describe Descartes circle configurations and Apollonian packings
\cite{kocik2021integral, kocik2020spinors, kocik2020apollonian, kocik2019spinors, kocik2007theorem, Kocik_2006}. 
However in those works, spinors are complex numbers (defined up to sign) describing tangencies between pairs of circles. This is significantly different from our approach. 
We also mention that the linear algebra and bilinear forms of Lagarias--Mallows--Wilks \cite{LagariasMallowsWilks} and Aharonov--Stephenson \cite{StephensonAharonov} are of a somewhat similar flavour.

\medskip

{\noindent \textbf{Structure of this paper.} } 
\refsec{background} provides relevant background forming the basis of this paper.  \refsec{flower_to_horocycles} demonstrates how the $n$-flower can be linked to horocycles in hyperbolic geometry. 
\refsec{horospheres_to_spinors} then equips these horocycles with spinor coordinates in $\mathbb{C}^2$ by applying a correspondence between spinors and horospheres. \refsec{spinor_calc} performs calculations on spinors, which are not necessary for but motivate the main result. 
Finally, \refsec{proof} proves the generalised Descartes theorem.

\medskip

{\noindent \textbf{Acknowledgments.}}
The authors are supported by Australian Research Council grant DP210103136.

\section{Background}
\label{Sec:background}

\subsection{Spinors and Horospheres} 
\label{Sec:spinors_horospheres}

We state here results of the first author in  \cite{Mathews_Spinors_horospheres} required in the sequel. 

In \cite{PenroseRindler}, Penrose and Rindler consider 2-component spinors $(\xi,\eta) \in \C^2$ and provide them with interpretations in Minkowski space, and more generally in relativity theory. They associate to such a spinor a \emph{null flag}, consisting of a ray on the future light cone, together with a spinorial tangent plane to the light cone. There is a natural bilinear form $\{ \cdot, \cdot \}$ on such spinors, given by the standard complex symplectic form on $\C^2$: given two spinors $\alpha = (\xi,\eta)$ and $\alpha ' = (\xi',\eta')$, 
\[
\{\alpha,\alpha'\} = 
\det \begin{pmatrix} \xi & \xi' \\ \eta & \eta' \end{pmatrix}
= \xi \eta' - \eta \xi'.
\]
In \cite{PennerPunctured}, Penner associates to each point on the future light cone a horocycle in the hyperboloid model of hyperbolic space. 
In \cite{Mathews_Spinors_horospheres}, the first author combined these associations to prove the following.
\begin{theorem}
\label{Thm:MathewsThm1}
There is an explicit, smooth, bijective, $SL(2,\C)$-equivariant correspondence between nonzero spinors $(\xi,\eta) \in \mathbb{C}^2$, and horospheres $H$ in hyperbolic 3-space decorated with spin-directions.
\end{theorem}

We briefly explain the notions in this theorem; see \cite[sec. 4]{Mathews_Spinors_horospheres} for further details.

A \emph{decoration} on a horosphere is a parallel oriented line field on it, i.e., roughly speaking, a direction along it. Such  fields or directions are well defined  because a horosphere is isometric to the Euclidean plane.
A decoration on a horosphere then lifts to two different \emph{spin decorations} as follows. A parallel oriented line field on a horosphere $H$ can be equivalently described by a parallel unit tangent vector field $v$ on $H$. From $v$ we can form \emph{inward} and \emph{outward} right-handed orthonormal frame fields on $H$ given by $f^{in} = (N^{in}, v, N^{in} \times v)$ and $f^{out} = (N^{out}, v, N^{out} \times v)$, where $N^{in}, N^{out}$ are respectively inward and outward normal vector fields on $H$. Each of  these two frame fields has two continuous lifts to the spin double cover of the frame bundle, which we call \emph{inward} and \emph{outward spin decorations}. We define an inward spin decoration $W^{in}$ and an outward spin decoration $W^{out}$ to be \emph{associated} if they are related by a specific rotation around $v$. A \emph{spin decoration} on $H$ is then defined as a pair $(W^{in}, W^{out})$ of associated inward and outward spin decorations. 

The action of $SL(2,\C)$ on $\C^2$ is by matrix-vector multiplication. The action of $SL(2,\C)$ on horospheres is a lift of the standard action of $PSL(2,\C)$ on hyperbolic 3-space by isometries (indeed, $PSL(2,\C) \cong \Isom^+ (\hyp^3)$, the orientation-preserving isometry group of hyperbolic 3-space). Each isometry in $PSL(2,\C)$ has two lifts to $SL(2,\C)$, which are negatives of each other. Just as with any universal cover, a lift of an isometry $\phi$ from $PSL(2,\C)$ to $SL(2,\C)$ may be specified by a path in $PSL(2,\C)$ from the identity to $\phi$; this path of isometries, applied to a spin-decorated horosphere, yields a path of spin-decorated horospheres; this determines the action of $SL(2,\C)$ on spin-decorated horospheres.

In the upper half space model, given in standard fashion as
\[
\mathbb{U}^3 = \{ (x,y,z) \; \mid \; z > 0 \} = \{ (x+yi, z) \; \mid \; z > 0 \} = \C \times \R_+
\]
with Riemannian metric $ds^2 = (dx^2 + dy^2 + dz^2)/z^2$, the explicit correspondence of \refthm{MathewsThm1} is particularly simple \cite[prop. 3.9]{Mathews_Spinors_horospheres}.  The horosphere $H$ corresponding to the spinor $(\xi,\eta)$ has centre $\xi/\eta$, the centre of a horosphere being its point at infinity. If $\xi/\eta \in \mathbb{C}$, then $H$ appears in the model as a Euclidean sphere of diameter $1/|\eta|^2$. A decoration on a horosphere can be given by a tangent direction to $H$ at a point; at the point of $H$ with maximum $z$-coordinate, tangent directions are parallel to $\C$ and are conveniently specified by complex numbers (the ``north pole specification" of a direction). The decoration on the horosphere $H$ corresponding to $(\xi,\eta)$ is north-pole specified by $i/\eta^2$. (The spinor $(-\xi, -\eta)$ yields the same decoration on the same horosphere, but a different spin decoration.)
If $\xi/\eta=\infty$ then $H$ appears in the model as a horizontal plane at Euclidean height $|\xi|^2$. Since such a horosphere appears parallel to $\C$, a direction can again be specified by an element of $\C$. The decoration on this $H$, corresponding to $(\xi,\eta)$, is then specified by $i\xi^2$.

We will be considering reducing this situation to 2 dimensions, as discussed in \cite[sec. 5]{Mathews_Spinors_horospheres}. The points $(x,y,z)$ of the upper half space model $\mathbb{U}^3$ with $y=0$ form the upper half plane model of the hyperbolic plane $\mathbb{U}^2$. The boundary at infinity of $\mathbb{U}^2$ is then given by $\partial \mathbb{U}^2 = \R \cup \{\infty\} \subset \C \cup \{\infty\} = \partial \mathbb{U}^3$. Horospheres in $\mathbb{U}^3$ centred at points of $\R \cup \{\infty\}$ correspond bijectively to horocycles in $\mathbb{U}^2$, with the bijection given by intersecting with $\mathbb{U}^2$. In this way we may regard horocycles $H$ of $\mathbb{U}^2$ as arising from horospheres $\widetilde{H}$ of $\mathbb{U}^3$ centred at $\R \cup \{\infty\}$.

Given a horocycle $H$ in $\mathbb{U}^2 \subset \mathbb{U}^3$, a \emph{planar spin decoration} on $H$  is a spin decoration on $\widetilde{H}$ whose decoration is (north-pole) specified by the direction $i$. A spinor $(\xi, \eta)$ describes a planar spin decoration on a horocycle if and only if $\xi$ and $\eta$ are both \emph{real} \cite[lem. 5.6]{Mathews_Spinors_horospheres}. Each horocycle $H$ of $\mathbb{U}^2$ has precisely two planar spin decorations. These two spin decorations are related to each other by a $2\pi$ rotation. The corresponding spinors are negatives of each other. Thus, roughly speaking, reducing to 2 dimensions amounts to considering \emph{real} spinors.

\subsection{Lambda Lengths}

Penner \cite{PennerPunctured} introduced the notion of $\lambda$-length between horocycles in the hyperbolic plane. Take the geodesic $\gamma$ connecting the centres of two horocycles and consider the segment external to both; if they overlap, then consider the segment within the overlap and call the distance negative. If the hyperbolic length of this segment is $\rho$, the $\lambda$-length is defined as $e^{\rho/2}$.

In \cite{Mathews_Spinors_horospheres} the first author generalised $\lambda$-lengths to 3 dimensions. Given two spin-decorated horospheres $H_1, H_2$ with distinct centres $z_1, z_2$, we may measure a \emph{complex} distance from $H_1$ to $H_2$ as follows. Let $\gamma$ be the oriented geodesic from $z_1$ to $z_2$. Let $\rho$ be the signed distance along $\gamma$ (using the orientation of $\gamma$) from $p_1 = \gamma \cap H_1$ to $p_2 = \gamma \cap H_2$. The spin decoration on $H_1$ contains an inward spin decoration $W^{in}_1$, and the spin decoration $H_2$ contains an outward spin decoration $W^{out}_2$. These spin decorations $W^{in}_1, W^{out}_2$ project to frames $f^{in}_1, f^{out}_2$ at $p_1, p_2$ respectively, whose first vectors are positively tangent to $\gamma$. The two frames are then related by a signed translation of length $\rho$ along $\gamma$, followed by a rotation of some angle $\theta$ about $\gamma$. Lifting to spin decorations, this $\theta$ is well defined modulo $4\pi$. The complex distance from $H_1$ to $H_2$ is then $\rho + i \theta$, and the \emph{complex $\lambda$-length} from $H_1$ to $H_2$ is
\[
\lambda_{12} 
= \exp\left(\frac{\rho + i \theta}{2}\right)
\]
It turns out that the complex $\lambda$-length is \emph{antisymmetric}: $\lambda_{12} = - \lambda_{21}$.
See \cite[sec. 4]{Mathews_Spinors_horospheres} for further details.

If we have two horocycles $H_1, H_2$ in $\mathbb{U}^2$ with planar spin decorations, then the frames $f^{in}_1, f^{out}_2$ have their first vectors pointing along $\gamma$ and their second vectors pointing the direction of the decoration, i.e. normal to $\mathbb{U}^2$ in the direction specified by $i$, or equivalently using the standard coordinates $(x,y,z)$, in the positive $y$-direction. It follows that $\theta = 0$ modulo $2\pi$. Then we observe that the complex $\lambda$-length from $H_1$ to $H_2$ is real, taking a positive value if $\theta = 0$ mod $4\pi$ and a negative value if $\theta = 2\pi$ mod $4\pi$. Penner's 2-dimensional $\lambda$-length is the positive value.

In \cite{Mathews_Spinors_horospheres}, the first author showed that complex $\lambda$-lengths are given by spinors using Penrose and Rindler's bilinear form.
\begin{theorem}
\label{Thm:MathewsThm2}
Consider two spinors $\kappa_1,\ \kappa_2$, corresponding to spin-decorated horospheres $H_1,H_2$. Then the complex $\lambda$-length $\lambda_{12}$ from $H_1$ to $H_2$ is given by
\[
\lambda_{12} = \{\kappa_1,\kappa_2\} .
\]
\end{theorem}

Indeed, in the case where $\kappa_1$, $\kappa_2$ are real spinors, describing horocycles with planar spin decorations, we observe that $\{\kappa_1,\kappa_2\}$ is real, as is $\lambda_{12}$. The complex $\lambda$-length $\lambda_{12}$ is positive or negative. Replacing one of the $\kappa_j$ with $-\kappa_j$ has the effect of changing the spin decoration on $H_j$, adding $2\pi$ to $\theta$ (mod $4\pi$) and changing the sign of $\lambda_{12}$ (by introducing a factor of $e^{i\pi} = -1$). Similarly, replacing $(\kappa_1, \kappa_2)$ with $(\kappa_2, \kappa_1)$ changes the sign of the complex $\lambda$-length.

Complex $\lambda$-lengths have several nice properties. For instance, given four horocycles (whose $\lambda$-lengths measure line segments forming an inscribed quadrilateral inside the disk), their $\lambda$-lengths satisfy Ptolemy's relation
\[\lambda_{12}\lambda_{34}+\lambda_{23}\lambda_{14}=\lambda_{13}\lambda_{24}.\]

Importantly for our purposes, the $\lambda$-length between two horocycles (with planar spin decorations) is $\pm 1$ if and only if the horocycles are tangent. This gives a simple condition for checking a purported $n$-flower.

\section{From Flowers to Horocycles}
\label{Sec:flower_to_horocycles}

Consider an $n$-flower ($n \geq 3$) in the Euclidean plane, as in \refdef{flower}. Let each circle $C_\bullet$ have radius $r_\bullet$ and curvature $\kappa_\bullet = 1/r_\bullet$.

Dilating the entire configuration by a factor of $\kappa_\infty$ sends the curvature $\kappa_\infty \mapsto 1$ and, for each petal circle, $\kappa_j \mapsto \kappa_j/\kappa_\infty$. We may thus assume the central circle has unit radius, so that $\kappa_\infty = 1$, and for each petal circle $C_j$ write $\kappa_j$ for the resulting curvature.

Invert the configuration in the unit circle $C_\infty$. Each $C_j$ is mapped to a circle $\mathring{C}_j$ internally tangent to $C_\infty$. For each $j$ mod $n$ then $\mathring{C}_j$ is externally tangent to $\mathring{C}_{j-1}$ and $\mathring{C}_{j+1}$. See \reffig{PoincareFlower}.  We denote by $\mathring{r}_j$ and $\mathring{\kappa}_j$ the (Euclidean) radius and curvature of $\mathring{C}_j$ respectively.
(The superscript ring notation is intended to indicate ``inside the unit disc".)

\begin{figure}[!ht]
\centering
\begin{tikzpicture}[scale=0.7]
\draw[very thick](-0.2,0) circle (1.35);
\draw[very thick](2.20,-0.31) circle (1.07);
\draw[very thick](1.57,2.02) circle (1.32);
\draw[very thick, red] (0.95,0.77) circle (2.76);
\end{tikzpicture}
\caption{\textit{An inverted $3$-flower.}}
\label{Fig:PoincareFlower}
\end{figure}
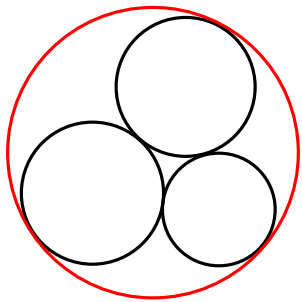

\begin{lemma}
\label{Lem:kappa_translation}
For each $j$,
\[
\mathring{\kappa}_j = \kappa_j +2.
\]
\end{lemma}

\begin{proof}
Let $O$ denote the centre of $C_\infty$. The furthest point $A$ from $O$ on $C_j$ has distance $1+2r_j$ from $O$. The closest point $A'$ to $O$ on $\mathring{C}_j$ has distance $1-2\mathring{r}_j$ from $O$. See \reffig{InversionCurvature}.

\begin{figure}[!h]
\begin{center}
\begin{tikzpicture}[scale=0.7]
\draw[very thick](1.82,0) circle (0.9);
\draw[very thick](4.76,0) circle (2);
\draw[very thick, red] (0,0) circle (2.76);
\filldraw[black] (0,0) circle (0.05) node [above left] {$O$};
\filldraw[black] (1.82,0) circle (0.05); 
\filldraw[black] (4.76,0) circle (0.05); 
\filldraw[black] (6.76,0) circle (0.06) node [right] {$A$};
\filldraw[black] (0.92,0) circle (0.06) node [below left] {$A'$};
\draw[dashed] (1.82,0) -- (2.76,0) node [midway, above] {$\mathring{r}$};
\draw[dashed] (4.76,0) -- (6.76,0) node [midway, above] {$r$};;
\end{tikzpicture}
\caption{\textit{Inversion of a petal circle.}}
\label{Fig:InversionCurvature}
\end{center}
\end{figure}
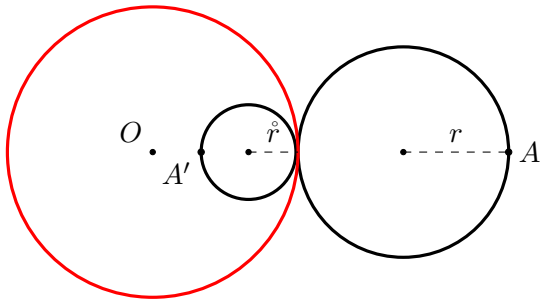

Inversion in $C_\infty$ sends $A \mapsto A'$, so 
\[
( 1+ 2r_j ) (1 - 2 \mathring{r}_j )=1,
\]
which upon substituting $r_j = 1/\kappa_j$ and $\mathring{r}_j = 1/\mathring{\kappa}_j$ yields the desired equation.
\end{proof}

We now consider the unit circle $C_{\infty}$ as the boundary of the disc model $\mathbb{D}^2$ of the hyperbolic plane. Each $\mathring{C}_j$ then appears as a horocycle. 

Next, we convert from the disc model $\mathbb{D}^2$ to the upper half plane model $\mathbb{U}^2$ of the hyperbolic plane. Regarding $\mathbb{D}^2$ and $\mathbb{U}^2$ as subsets of the complex plane in the standard way, the two models are related by the Cayley transform $\mathfrak{S} \colon \mathbb{U}^2 \To \mathbb{D}^2$ and its inverse:
\begin{equation}
\label{Eqn:Cayley_transform}
\mathfrak{S}(z) = \frac{z-i}{z+i}, \quad
\mathfrak{S}^{-1}(z) = \frac{(z+1)i}{1-z}.
\end{equation}
Observe that $\mathfrak{S}^{-1}$ sends $\partial \D^2$ to $\partial \U^2 = \R \cup \{\infty\}$, and sends horocycles in $\D^2$ to horocycles in $\U^2$.
Moreover, the point $1$ on $\partial \mathbb{D}^2$ corresponds to the point $\infty$ on $\partial \mathbb{U}^2 = \R \cup \{\infty\}$. And as the unit circle $\partial \mathbb{D}^2$ is traversed anticlockwise, its image $\R \cup \{\infty\}$ under $\mathfrak{S}^{-1}$ is traversed in the increasing direction.

Before applying $\mathfrak{S}^{-1}$, if necessary we reflect the configuration in an arbitrary hyperbolic line through the origin, so that the centres of $\mathring{C}_0, \mathring{C}_1, \ldots, \mathring{C}_{n-1}$ are in anticlockwise order around $\partial \mathbb{D}^2$. This is a reflection in both hyperbolic and Euclidean geometry, so preserves all $\mathring{r}_j$ and $\mathring{\kappa}_j$. Then, if necessary, we rotate $\mathbb{D}^2$ about $0$ so that, proceeding anticlockwise from $1$ around $\partial \mathbb{D}^2$, we encounter the centre of $\mathring{C}_0$ first, then $\mathring{C}_1, \ldots, \mathring{C}_{n-1}$, in  order. Again, this is a rotation in both hyperbolic and Euclidean geometry, so preserves all $\mathring{r}_j$ and $\mathring{\kappa}_j$.

Now apply $\mathfrak{S}^{-1}$ to the entire configuration. Since no $\mathring{C}_j$ is tangent to $\partial \mathbb{D}^2$ at $1$, then no $\overline{C}_j$ is tangent to $\partial \mathbb{U}^2 = \R \cup \{\infty\}$ at $\infty$. Thus each horocycle $\overline{C}_j$ has centre in $\R \subset \partial \U^2$, and appears as a circle in $\mathbb{U}^2$, with a finite radius $\overline{r}_j$ and curvature $\overline{\kappa}_j$. Moreover, since the centres of $\mathring{C}_j$ are in order around $\partial \D^2$ proceeding anticlockwise from $1$, the centres of the $\overline{C}_j$ are in increasing order along $\R \subset \partial \U^2$. 

We thus obtain a ``flat $n$-flower" where the original central circle has become $\R \cup \{\infty\}$, and each petal circle has become a horocycle in $\U^2$, with each $\overline{C}_j$ externally tangent to $\overline{C}_{j-1}$ and $\overline{C}_{j+1}$ (with indices taken mod $n$), and the horocycles 
increasing in index from left to right.
See \reffig{UpperHalfFlower}. 
(The overline notation is intended to indicate ``in a flat flower along the real line".)

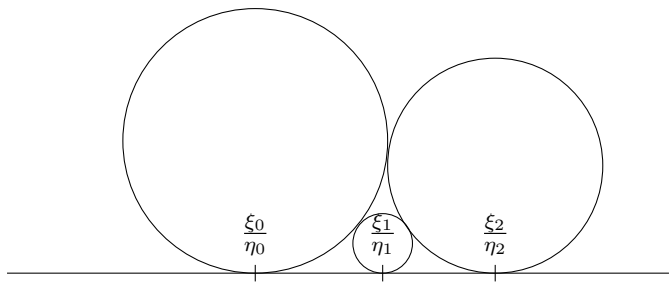
\begin{figure}[!ht]
\begin{center}
\begin{tikzpicture}[scale=1.1]
\draw[-] (-4,0) -- (4,0);
\draw (-1,1.6) circle (1.6);
\draw (1.9,1.3) circle (1.3);
\draw (0.54,0.36) circle (0.36);
\draw (-1,-0.1) -- (-1,0.1) node [above] {$\frac{\xi_0}{\eta_0}$};
\draw (0.54,-0.1) -- (0.54,0.1) node [above] {$\frac{\xi_1}{\eta_1}$};
\draw (1.9,-0.1) -- (1.9,0.1) node [above] {$\frac{\xi_2}{\eta_2}$};
\end{tikzpicture}
\caption{\textit{A flat $3$-flower}}
\label{Fig:UpperHalfFlower}
\end{center}
\end{figure}

Finally, observe that the circles $\overline{C}_j$ can be translated horizontally without affecting the tangencies of circles or their or curvatures. We apply such a translation so that $\overline{C}_0$ is tangent to the real line at $0$. Since the horocycles $\overline{C}_j$ are tangent to the real line in increasing order, for $j \geq 1$ each $\overline{C}_j$ is tangent to the real line at a positive number.

\section{From Horospheres to Spinors}
\label{Sec:horospheres_to_spinors}

We now introduce spinors for the horocycles $\overline{C}_j$. 

As discussed in \refsec{spinors_horospheres}, each horocycle  in the hyperbolic plane has two planar spin decorations,  corresponding to two real spinors which are negatives of each other.

Thus each horocycle $\overline{C}_j$ corresponds to a real spinor $ \alpha_j = (\xi_j, \eta_j)$, well defined up to sign. The centre of $\overline{C}_j$ is at $\xi_j/\eta_j$, which lies in $\R$, so each $\eta_j$ is nonzero. The Euclidean radius $\overline{r}_j$ and curvature $\overline{\kappa}_j$ of $\overline{C}_j$ are given by
\begin{equation}
\label{Eqn:eta_kappa}
\overline{r}_j = \frac{1}{2 \eta_j^2}
\quad \text{and} \quad
\overline{\kappa}_j = 2 \eta_j^2.
\end{equation}
(In general the Euclidean diameter of the horosphere corresponding to $(\xi, \eta)$ is $\frac{1}{|\eta|^2}$. Here $\eta$ is real.)

We now choose each $\alpha_j$ so that $\eta_j$ is positive.

Recall $\overline{C}_0$ is constructed to be tangent to the real line at $0$, and the centres of the $\overline{C}_j$ are in increasing order along the real line. Thus we have
\[
0 = \frac{\xi_0}{\eta_0} < \frac{\xi_1}{\eta_1} < \cdots < \frac{\xi_{n-1}}{\eta_{n-1}}.
\]
Thus $\xi_0 = 0$ and, since all $\eta_j$ are chosen positive, for $j \geq 1$ we have $\xi_j > 0$.

Calculating bilinear forms, for all $0 \leq j<k \leq n-1$ we have
\[
\{ \alpha_j, \alpha_k \}
= \xi_j \eta_k - \xi_k \eta_j = \eta_j \eta_k \left( \frac{\xi_j}{\eta_j} - \frac{\xi_k}{\eta_k} \right) < 0.
\]
(This is the opposite of the \emph{totally positive} notion of \cite[sec. 6]{Mathews_Spinors_horospheres}; the $\alpha_j$ here are ``totally negative".)

Now for each $j$ mod $n$ we know that $\overline{C}_j$ and $\overline{C}_{j+1}$ are tangent. Thus the complex lambda length between planar spin decorations of $\overline{C}_j$ and $\overline{C}_{j+1}$ is $\pm 1$. Combining these observations we have the following.
\begin{lem}
\label{Lem:spinor_signs}
Choosing the spinors $\alpha_j = (\xi_j, \eta_j)$ for planar spin decorations on each $\overline{C}_j$ so that all $\eta_j>0$, then for all $0 \leq j \leq n-2$ we have
\[
\{\alpha_j, \alpha_{j+1} \} = \xi_j \eta_{j+1} - \xi_{j+1} \eta_j = -1,
\]
and $\{\alpha_0, \alpha_{n-1} \} = -1$.
\qed
\end{lem}

The following proposition gives a useful relationship between each spinor $(\xi_j, \eta_j)$, and the Euclidean curvature $\mathring{\kappa}_j$ of the corresponding horocycle in the disc model.
\begin{prop}
\label{Prop:kappa_equation}
Let $(\xi, \eta)$ be a real spinor and $H$ the corresponding horocycle in $\D^2$ with a planar spin decoration. Let the Euclidean curvature of $H$ as it appears in the disc model $\D^2$ be $\mathring{\kappa}$. Then
\begin{equation}
\label{Eqn:kappaquation}
\mathring{\kappa} = \xi^2+\eta^2+1.
\end{equation}
\end{prop}
Combining \refprop{kappa_equation} with \reflem{kappa_translation}, we observe that $\kappa = \mathring{\kappa}-2 = \xi^2 + \eta^2 - 1$, so the curvatures of the original circles are also usefully expressed in terms of spinor coordinates.

In order to prove \refprop{kappa_equation} we use some lemmas.
\begin{lemma}
\label{Lem:rotate_horocycle_U2}
Let $H$ be a horocycle in $\mathbb{U}^2$  with a planar spin decoration, corresponding to the spinor $(\xi,\eta)$. Then the planar spin-decorated horocycle $H'$ obtained by rotating $H$ by $\pi$ about $i \in \U^2$
 corresponds to the spinor $(-\eta,\xi)$.
\end{lemma}
Let us briefly explain what is meant by rotating $H$ in the above proposition. As $H$ has a planar spin decoration, it corresponds to a spin-decorated horosphere $\widetilde{H}$ in $\U^3$. This $\widetilde{H}$ has a spin decoration $W$ consisting of associated inward and outward spin decorations, which are lifts of frame fields $f$ along $\widetilde{H}$. An orientation-preserving isometry $\phi \in \Isom^+(\U^3) \cong PSL(2,\C)$ may be applied to the horosphere $\widetilde{H}$ and (via its derivative) the frame fields $f$, yielding a decoration on a horosphere $\widetilde{H'}$. If we specify a lift of $\phi$ to the spin double cover $SL(2,\C)$, then there is a well-defined map of the spin decorations $W$ from $\widetilde{H}$ to $\widetilde{H'}$. A lift of $\phi$ to $SL(2,\C)$ can be specified by a path from the identity to $\phi$ in $PSL(2,\C)$. Here, we have a rotation of $\pi$ about $i \in \U^2$, which naturally extends to a rotation in the 3-dimensional model $\U^3$ about the geodesic normal to $\U^2$ through $i$. This rotation is naturally lifted to the spin double cover of $\Isom^+(\U^3)$ by taking the path of isometries $\phi_t$ over $t \in [0,1]$, where $\phi_t$ is a rotation of angle $\pi t$ about the same axis as $\phi$. The isometries $\phi_t$ then take the spin decoration on $\widetilde{H}$ to a spin decoration on $\widetilde{H'}$, corresponding to the planar spin-decorated $H'$ of the lemma.

\begin{proof}
Consider the matrices $M(\theta)$ for $\theta \in \R$ given by
\[
M(\theta) = \begin{pmatrix}
    \cos(\theta) & -\sin(\theta) \\
    \sin(\theta) & \cos(\theta) \\
\end{pmatrix}
\in SL(2,\R) \subset SL(2,\C).
\]
Each $M(\theta)$ describes an orientation-preserving isometry of $\U^2$ or $\U^3$ as a M\"{o}bius transformation: a rotation of $2\theta$ about $i$ in $\U^2$, or a rotation of $2\theta$ about the geodesic normal to $\U^2$ through $i$ in $\U^3$. Then $M(\frac{t\pi}{2})$, over $t \in [0,1]$, describes a path of isometries from the identity to the rotation described in the lemma. As in fact all $M(\theta)$ lie in $SL(2,\C)$, the double cover of the isometry group $PSL(2,\C)$, then $M(\frac{\pi}{2})$ is the desired spin lift of the isometry described in the lemma: it is a  rotation of $\pi$ (not $-\pi$) about $i$.

We now use the equivariance of \refthm{MathewsThm1}. Since $(\xi,\eta)$ corresponds to $H$, and $H'$ is obtained from $H$ by applying $M(\pi/2)$, then $H'$ corresponds to the spinor
\[
M \left( \frac{\pi}{2} \right) \begin{pmatrix} \xi \\ \eta \end{pmatrix}
= \begin{pmatrix} 0 & -1 \\ 1 & 0 \end{pmatrix} \begin{pmatrix} \xi\\ \eta \end{pmatrix}
= \begin{pmatrix} -\eta \\ \xi \end{pmatrix}.
\]
\end{proof}

Translating \reflem{rotate_horocycle_U2} into the disc model $\D^2$ via the Cayley transform \refeqn{Cayley_transform}, we immediately obtain the following.
\begin{lem}
\label{Lem:rotate_horocycle_D2}
Let $H$ be a horocycle in $\mathbb{D}^2$  with a planar spin decoration, corresponding to the spinor $(\xi,\eta)$. Then the planar spin-decorated horocycle $H'$ obtained by rotating $H$ by $\pi$ about $0 \in \D^2$ corresponds to the spinor $(-\eta,\xi)$.
\qed
\end{lem}

\begin{proof}[Proof of \refprop{kappa_equation}]
Let $H'$ be the planar spin-decorated horocycle obtained by rotating $H$ by $\pi$ about $0 \in \D^2$. This is a Euclidean rotation as well as a hyperbolic rotation, so $H'$ has the same radius and curvature as $H$. By \reflem{rotate_horocycle_D2}, $H'$ corresponds to the spinor $(-\eta, \xi)$.

Let $X$ be the hyperbolic distance from $0 \in \D^2$ to $H$. By symmetry, the hyperbolic distance from $0$ to $H'$ is also $X$, so the hyperbolic distance from $H$ to $H'$ is $2X$. The metric in $\D^2$ is given by
\[ds^2 = \frac{4\ dr^2}{(1-r^2)^2},\]
where $r$ is a Euclidean radial coordinate. Since $H$ and $H'$ have Euclidean radius $\mathring{r}$, the line from $0$ to $H$ is a Euclidean line from $r=0$ to $r=1-2\mathring{r}$. Thus
\[
X = \int_0^{1-2\mathring{r}} \frac{2 dr}{1-r^2} = 
\ln \left( \frac{1-\mathring{r}}{\mathring{r}} \right).
\]

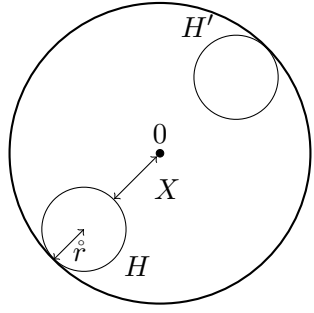
\begin{figure}[!ht]
\begin{center}
\begin{tikzpicture}[scale=1]
\draw[thick] (0,0) circle (2);
\draw[rotate=45] (1.43, 0) circle (0.56);
\draw[<->,rotate=45] (-1.43,0) -- (-2,0) node[right,pos=0.7] {$\mathring{r}$};
\draw[rotate=45] (-1.43,0) circle (0.56);
\draw[rotate=45,<->] (-2+1.14,0) -- (-0.06,0) node[anchor=north west,pos=0.7] {$X$}; 
\filldraw[black] (0,0) circle (0.05) node[above] {$0$};
\draw (-0.3,-1.5) node {$H$};
\draw (0.5,1.7) node {$H'$};
\end{tikzpicture}
\caption{\textit{Computing the radius of $H$}}
\end{center}
\end{figure}

Let us now consider the complex lambda length from $H$ to $H'$. We thus consider the 3-dimensional spin-decorated horospheres $\widetilde{H}, \widetilde{H'}$ corresponding to $H,H'$. The distance between $\widetilde{H}$ and $\widetilde{H'}$ is $\rho = 2X$. Consider the angle $\theta$ between the inward spin frame $W_{in}$ of $\widetilde{H}$ and the outward spin frame $W'_{out}$ of $\widetilde{H'}$ along the common perpendicular from $H$ to $H'$. Since both spin decorations are planar, we have $\theta = 0$ or $2\pi$ mod $4\pi$. Following \cite[defn. 4.3]{Mathews_Spinors_horospheres}, the outward spin frame associated to $W_{in}$ is obtained from $W_{in}$ by a rotation of angle $-\pi$ about the decoration direction normal to $\D^2$. The outward spin decoration of $W'_{out}$ is obtained from this associated frame by a rotation of angle $\pi$ about the same normal direction to $\D^2$ (i.e. the rotation which takes $\widetilde{H}$ to $\widetilde{H'}$). We conclude that $\theta = 0$. Thus the complex lambda length $\lambda$ from $H$ to $H'$ is
\begin{equation}
\label{Eqn:lambda_eqn_1}
\lambda = \exp \left( \frac{\rho + i \theta}{2} \right)
= e^{X} = \frac{1-\mathring{r}}{\mathring{r}}.
\end{equation}

On the other hand, applying \refthm{MathewsThm2} to the two spinors involved yields
\begin{equation}
\label{Eqn:lambda_eqn_2}
\lambda = \{(\xi,\eta),(-\eta,\xi)\}= \xi^2 + \eta^2.
\end{equation}
Equating \refeqn{lambda_eqn_1} and \refeqn{lambda_eqn_2} and using $\mathring{\kappa} = 1/\mathring{r}$ then gives the desired result.
\end{proof}

\section{Spinor Calculations}
\label{Sec:spinor_calc}

We now present some calculations that are not strictly necessary for the main result, but may be of interest, and may motivate some expressions arising in the main result. The reader interested only in the proof of the main result may skip this section.

\subsection{Equation Relating $\eta_j$}

It follows from \reflem{spinor_signs} that, for $0 \leq j \leq n-2$,
\begin{equation}
\label{Eqn:BilinearRelation}
\frac{\xi_{j+1}}{\eta_{j+1}} - \frac{\xi_j}{\eta_j} =\ \frac{1}{\eta_j \eta_{j+1}},
\end{equation}
which is the distance along the real line between the centres of the horospheres $\overline{C}_j$ and $\overline{C}_{j+1}$.
Summing these distances over all $j$ from $0$ to $n-2$, we obtain
\begin{equation}
\label{Eqn:telescope}
\frac{\xi_{n-1}}{\eta_{n-1}} - \frac{\xi_0}{\eta_0}
=
\sum_{j=0}^{n-2} \frac{1}{\eta_j \eta_{j+1}}.
\end{equation}
Since $\{\alpha_0, \alpha_{n-1} \} = \xi_0 \eta_{n-1} - \xi_{n-1} \eta_0 = -1$, we also have 
\begin{equation}
\label{Eqn:basicetarelation}
\sum_{j=0}^{n-2} \frac{1}{\eta_j \eta_{j+1}}  = \frac{1}{\eta_0\eta_{n-1}}.
\end{equation}

\subsection{Equation Relating $\overline{\kappa}_j$}

Substituting $\eta_j$ for $\overline{\kappa}_j$ in \refeqn{basicetarelation} using \refeqn{eta_kappa} yields an equation relating the curvatures $\overline{\kappa}_j$:
\begin{equation}
\label{flatequation}
\sum_{j=0}^{n-2} \frac{1}{\sqrt{\overline{\kappa}_j\overline{ \kappa_{j+1}}}}  = \frac{1}{\sqrt{\overline{\kappa}_0 \overline{\kappa}_{n-1}}}.
\end{equation}
We can obtain  a polynomial relation between  the $\overline{\kappa}_j$ by clearing denominators and squaring out roots. 

This is a relation between the curvatures in a flat $n$-flower, i.e. when $\kappa_\infty = 0$.

For example in the case $n=3$, the equation
\[
\frac{1}{\sqrt{\overline{\kappa}_0 \overline{\kappa}_1}} + \frac{1}{\sqrt{\overline{\kappa}_1 \overline{\kappa}_2}} = \frac{1}{\sqrt{\overline{\kappa}_0\overline{\kappa}_2}}
\]
leads to the polynomial
\[
\overline{\kappa}_0^2 + \overline{\kappa}_1^2 + \overline{\kappa}_2^2 -2\overline{\kappa}_0\overline{\kappa}_1 -2\overline{\kappa}_1\overline{\kappa}_2 - 2\overline{\kappa}_2\overline{\kappa}_0=0,
\]
which is equivalent to Descartes' equation \refeqn{Descartes} with $\kappa_{\infty}=0$.

\subsection{Recursive Computation of Spinor Coordinates}
\label{Sec:recursive_computation}

We show that, starting from $(\xi_0, \eta_0)$, the remaining $(\xi_j, \eta_j)$ can be calculated recursively using previous $(\xi_j, \eta_j)$ and the Euclidean curvatures $\mathring{\kappa}_j$ in the disc model. (Using \reflem{kappa_translation}, one could use the $\kappa_j$ instead of the $\mathring{\kappa}_j$.)

\begin{lemma}
The spinors $(\xi_j,\eta_j)$ satisfy the following:
\[
\xi_0 = 0, \quad \eta_0 = \sqrt{\mathring{\kappa}_0-1},
\]
and for $0 \leq j \leq n-2$,
\begin{align}
\label{Eqn:xi2}
\xi_{j+1} &= \left(\frac{-\xi_j + \eta_j\sqrt{(\mathring{\kappa}_{j+1}-1)(\mathring{\kappa}_j-1)-1}}{\eta_j(\mathring{\kappa}_j-1)} \right) \xi_j + \frac{1}{\eta_j },  \\
\label{Eqn:eta}
\eta_{j+1} &= \frac{-\xi_j + \eta_j\sqrt{(\mathring{\kappa}_{j+1}-1)(\mathring{\kappa}_j-1)-1}}{\mathring{\kappa}_j-1}.
\end{align}
\end{lemma}
By \reflem{kappa_translation} each $\mathring{\kappa}_j > 2$, so each quantity under a square root sign is positive, and we take the positive square root in each case. 

\begin{proof}
Since $\overline{C}_0$ is tangent to the real line at $\xi_0/\eta_0 = 0$, we have $\xi_0 = 0$.
From \refprop{kappa_equation} then  $\mathring{\kappa}_0 = \eta_0^2 + 1$, so $\eta_0^2 = \mathring{\kappa}_0 - 1$. Since all $\eta_j$ are taken to be positive then $\eta_0$ is as claimed.

From \refeqn{BilinearRelation} we have
\begin{equation}
\label{Eqn:xi}
\xi_{j+1} = \frac{\eta_{j+1}}{\eta_j}\xi_j+\frac{1}{\eta_j},
\end{equation}
which shows that \refeqn{xi2} follows from \refeqn{eta}. 

Squaring \refeqn{xi} and substituting the resulting expression for $\xi_{j+1}^2$ into 
\refeqn{kappaquation} for $\mathring{\kappa}_{j+1}$ yields
\[
\mathring{\kappa}_{j+1} = \eta_{j+1}^2 + \frac{\eta_{j+1}^2}{\eta_j^2} \xi_j^2 + 2\frac{\eta_{j+1}}{\eta_j^2} \xi_j+ \frac{1}{\eta_j^2}+1.
\]
This is a quadratic equation for $\eta_{j+1}$; solving in the standard way and using \refeqn{kappaquation} for $\mathring{\kappa}_j$, we obtain
\[
\eta_{j+1} = \frac{-\xi_j \pm  \eta_j\sqrt{(\mathring{\kappa}_{j+1}-1)(\mathring{\kappa}_j-1)-1}}{\mathring{\kappa}_j-1}.
\]
As all $\xi_j, \eta_j$, and $\mathring{\kappa}_j - 1$ are non-negative, we must take the $+$ in the $\pm$ in order for $\eta_{j+1}$ to be positive, yielding \refeqn{eta} as desired.
\end{proof}

\subsection{Variables $m_j$}

The square root expressions in the recursive equations \refeqn{xi2} and \refeqn{eta} for $\xi$ and $\eta$ involve square root expressions which arise throughout; they are in fact the $m_j$ of our main theorem. By \reflem{kappa_translation} we have $\mathring{\kappa}_j-1 = \kappa_j + 1$, so the two expressions given in the definition below are equal.
\begin{definition} 
\label{Def:m_j}
We define
\[
m_0 = \sqrt{\mathring{\kappa}_0 -1} = \sqrt{\kappa_0 + 1}
\]
and for $1 \leq j \leq n-1$,
\[
\label{newVariables}
m_j = \sqrt{ \left( \mathring{\kappa}_j - 1 \right) \left( \mathring{\kappa}_{j-1}-1 \right) -1 }
= \sqrt{ \left( \kappa_j + 1 \right) \left( \kappa_{j-1} + 1 \right) - 1}.
\]
\end{definition}

From this definition, a straightforward induction allows us to express each $\mathring{\kappa}_j$ or $\kappa_j$ in terms of the $m_j$; we state this in the following lemma and omit the proof. We observe that products arising here also appear in the main theorem. As usual, the empty product is taken to be $1$.
\begin{lemma}
\label{Lem:mKappaLemma}
For $j>0$ we have 
\begin{equation}
\kappa_j + 1 = \mathring{\kappa}_j-1 = \begin{cases}
\frac{m_0^2 \prod_{k=1}^{\frac{j}{2}} (m_{2k}^2+1)}{\prod_{k=1}^{\frac{j}{2}} (m_{2k-1}^2+1)} \text{, if j is even and }j>0\\
\frac{\prod_{k=0}^{\frac{j-1}{2}} (m_{2k+1}^2+1)}{m_0^2 \prod_{k=1}^{\frac{j-1}{2}} (m_{2k}^2+1)} \text{, if j is odd and }j>0.
\end{cases}
\end{equation}
\qed
\end{lemma}
For convenience, we express the above expansion of $\kappa_j+1$ in terms of the $m_k$ as $g_j$.

\subsection{Closed Form for Spinor Coordinates}

Between them, \refeqn{xi2} and \refeqn{eta} allow us to iteratively find $(\xi_j,\eta_j)$ in terms of the $\mathring{\kappa}_j$, or, using \refdef{m_j}, in terms of the $m_j$. In this section we describe closed forms for $\xi_j$ and $\eta_j$ in terms of the $m_j$. 

Since 
$\overline{C}_0$ is tangent to the real line at $0$, so that $\xi_0 = 0$, then forming a telescoping sum from \refeqn{BilinearRelation}, as in  \refeqn{telescope}, we have
\begin{equation} 
\label{Eqn:XiEquation}
\xi_j = \frac{1}{\eta_{j-1}} + \eta_j \sum_{k=1}^{j-1} \frac{1}{\eta_{k-1}\eta_k},
\end{equation}
so it suffices to find a closed form for $\eta$.

\begin{lemma}[A Closed Form for $\eta_j$]
\label{Lem:EtaLemma}
\[
\eta_j = \begin{cases}
\frac{m_0\left(\prod_{n=1}^j(m_n-i)+\prod_{n=1}^j(m_n+i)\right)}{2\prod_{n=1}^{\frac{j}{2}}\left(m_{2n-1}^2+1\right)}, & \text{ if j is even}\\
\frac{\left(\prod_{n=1}^j(m_n-i)+\prod_{n=1}^j(m_n+i)\right)}{2m_0\prod_{n=1}^{\frac{j-1}{2}}\left(m_{2n}^2+1\right)}, & \text{ if j is odd.}
\end{cases}
\]
\end{lemma}

\begin{proof} 
Substituting $m_j$ into \refeqn{eta} 
yields
\[
\eta_{j+1} = \frac{-\xi_j + \eta_j m_{j+1}}{\mathring{\kappa}_j -1}.
\]
Replacing $\xi_j$ using \refeqn{XiEquation}, and using the definition of $g_j$ after \reflem{mKappaLemma} gives
\begin{equation}
\label{Eqn:etaformB}
\eta_{j} = \frac{-1 + \eta_{j-1}\eta_{j-2}\left(m_j- \sum_{n=1}^{j-2} \frac{1}{\eta_{n-1}\eta_n}\right)}{\eta_{j-2}g_{j-1}}.
\end{equation}
Rearrange to solve for the sum:
\[
\sum_{n=1}^{j-2} \frac{1}{\eta_{n-1}\eta_n} = m_j - \frac{\eta_{j} \eta_{j-2}g_{j-1} +1}{\eta_{j-1}\eta_{j-2}}.
\]
Then, taking one term out of the sum, we have
\[
\sum_{n=1}^{j-2} \frac{1}{\eta_{n-1}\eta_n} = \sum_{n=1}^{j-3} \frac{1}{\eta_{n-1}\eta_n} + \frac{1}{\eta_{j-3}\eta_{j-2}}
= m_{j-1} - \frac{\eta_{j-1} g_{j-2}}{\eta_{j-2}}.
\]
Substituting this expression for the sum back into \refeqn{etaformB}, we obtain a recursive relation on $\eta_j$:
\[
\eta_j = \frac{-1+\eta_{j-1}^2g_{j-2}+\eta_{j-1}\eta_{j-2}(m_j-m_{j-1})}{\eta_{j-2}g_{j-1}}.
\]
The closed form is then obtained by induction,
taking $\eta_0=m_0$, $\eta_1=m_1/m_0$ for the base cases. 
\end{proof}

\subsection{Polynomial Relations Between $m_j$}

Multiplying \refeqn{basicetarelation} by $\eta_0\eta_1...\eta_{n-1}$ to clear denominators yields
\[
\left( \sum_{j=0}^{n-2} \eta_0\eta_1 \cdots \hat{\eta}_j\hat{\eta}_{j+1} \cdots \eta_{n-1} \right)-\eta_1\eta_2...\eta_{n-2}=0,
\]
where the hats indicate excluding those terms from the product. Substitution using  \reflem{EtaLemma} yields polynomial relations among the $m_j$, and these polynomials contain as factors the equations of our main result.

For example, when $n=3$ and $n=4$ we obtain, respectively,
\begin{gather*}
\frac{m_1 \left(m_1 m_0^2+m_2 m_0^2-m_1^2-1\right)}{m_0 \left(m_1^2+1\right)}=0, \\
\frac{m_1 \left(m_1 m_2-1\right) \left(-m_2^2+m_1 m_2+m_3 m_2+m_1 m_3-2\right)}{\left(m_1^2+1\right) \left(m_2^2+1\right)}=0,
\end{gather*}
and we expect that only the factor containing all the curvatures can be zero in general. Indeed, in the $n=3$ case this factor is an instance of \refeqn{mquationOdd} (or \refeqn{mquationOdd_alt}) from the  main theorem; and in the $n=4$ case an instance of \refeqn{mquationEven} (or \refeqn{mquationEven_alt}). This is how the equations of the main theorem were found.

\section{Proof of the Generalised Descartes Theorem}
\label{Sec:proof}

After the constructions of \refsec{flower_to_horocycles} and \refsec{horospheres_to_spinors}, we have real spinors $(\xi_j, \eta_j)$ for $0 \leq j \leq n-1$, describing horocycles with planar spin decorations arising from an $n$-flower. 

For each $j$, we define a complex number
\[
z_j = \xi_j + i \eta_j
\]
(here $i^2 = -1$ as usual).

From \reflem{spinor_signs} we have, for $1 \leq j \leq n-1$,
\[
\xi_{j-1}\eta_{j} - \xi_{j}\eta_{j-1} = -1,
\quad \text{or equivalently,} \quad
\Im \left( z_{j-1} \, \overline{z_{j}} \right) = 1. 
\]
which, being the nonzero component of the cross product $(\xi_{j-1},\eta_{j-1},0) \times (\xi_{j},\eta_{j},0)$, means that the $z_j$ occur in clockwise order around $0$ in the complex plane and each successive pair $z_{j-1}, z_{j}$ spans a parallelogram of area $1$. Since we chose $\xi_0 = 0$ and all other $\xi_j, \eta_j>0$ in \refsec{horospheres_to_spinors}, we have $z_0$ on the positive imaginary axis, and all the other $z_j$ lying in the top right quadrant of the complex plane as in \reffig{Parallelograms}. 
The sequence $\arg z_j$ for $0 \leq j \leq n-1$ is thus strictly decreasing in $(0, \pi/2]$.

Moreover, \reflem{spinor_signs} says that $\xi_0 \eta_{n-1} - \xi_{n-1} \eta_0 = \Im \left( z_0 \, \overline{z_{n-1}} \right) = -1$, so the first and last of the $z_j$ also span a parallelogram of area 1.
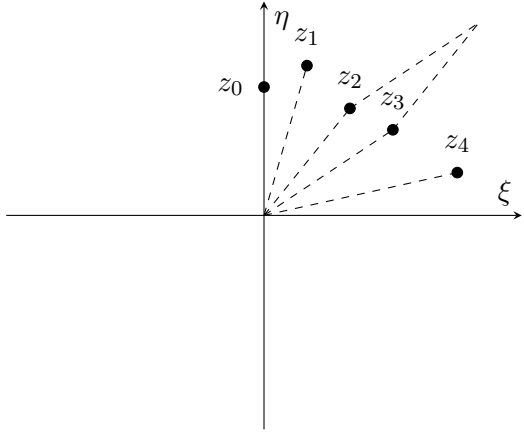
\begin{figure}
\centering
\begin{tikzpicture}
\begin{axis}[
    axis lines=middle,
    xmin=-12, xmax=12,
    ymin=-10, ymax=10,
    xtick=\empty, ytick=\empty, ylabel={$\eta$}, xlabel={$\xi$}
]
\addplot [only marks] table {
0   6
2   7
4   5
6   4
9   2
};

\node[label={180:{$z_0$}}] at (axis cs:0,6) {};
\node[label={90:{$z_1$}}] at (axis cs:2,7) {};
\node[label={90:{$z_2$}}] at (axis cs:4,5) {};
\node[label={90:{$z_3$}}] at (axis cs:6,4) {};
\node[label={90:{$z_4$}}] at (axis cs:9,2) {};
\addplot [domain=0:2, samples=2, dashed] {3.5*x};
\addplot [domain=0:4, samples=2, dashed] {(5/4)*x};
\addplot [domain=0:6, samples=2, dashed] {(4/6)*x};
\addplot [domain=0:9, samples=2, dashed] {(2/9)*x};

\addplot [domain=6:10, samples=2, dashed] {(5/4)*x-3.5};
\addplot [domain=4:10, samples=2, dashed] {(4/6)*x+7/3};
\end{axis}
\end{tikzpicture}
\caption{Parallelograms formed by the $z_j$.}
\label{Fig:Parallelograms}
\end{figure}

From \refprop{kappa_equation} (and the remark afterward applying \reflem{kappa_translation}) we have 
\[
\kappa_j = \xi_j^2 + \eta_j^2 - 1 = |z_j|^2 - 1
\]
so the problem of finding a relation among the $\kappa_j$ (i.e. generalising Descartes' theorem) is reduced to a plane Euclidean geometry problem: given vectors spanning parallelograms of area 1 in clockwise order around the origin as described above, find a relation among the lengths of those vectors.

We now  write the $m_j$ (\refdef{m_j}) for $1 \leq j \leq n-1$ in terms of the $z_j$: 
\begin{equation}
\label{Eqn:m_from_z}
m_j = \sqrt{ \left( \kappa_j + 1 \right) \left( {\kappa}_{j-1}+1 \right) -1} 
= \sqrt{|z_j|^2|z_{j-1}|^2-1}
\end{equation}

Since $\Im \left( z_{j-1} \overline{z_j} \right) = 1$, it follows that $\Re \left( z_{j-1} \overline{z_j} \right) = \pm m_j$.
But we saw above that the $\arg z_j$ form a decreasing sequence in $(0, \pi/2]$, so $0 < \arg (z_{j-1} \overline{z_j}) < \pi/2$. Hence each $z_{j-1} \overline{z_j}$ has positive real part, and we obtain
\[
z_{j-1} \overline{z_j} = m_j + i.
\]

We now observe that the desired equations \refeqn{mquationOdd} and \refeqn{mquationEven} contain the products
\begin{align*}
\prod_{j=1}^{n-1} \left( m_j - i \right) &= \prod_{j=1}^{n-1} \overline{z_{j-1}} z_j = \overline{z_0} \, \left| z_1 \cdots z_{n-2} \right|^2 \, z_{n-1}, \\
\prod_{j=1}^{n-1} \left( m_j + i \right) &= \prod_{j=1}^{n-1} z_{j-1} \overline{z_j} = z_0 \, \left| z_1 \cdots z_{n-2} \right|^2 \, \overline{z_{n-1}}.
\end{align*}
In the case of even $n$ then we find
\begin{align*}
\frac{i}{2} \left( \prod_{j=1}^{n-1} \left( m_j - i \right) - \prod_{j=1}^{n-1} \left( m_j + i \right) \right)
&= \frac{i}{2} \left| z_1 \cdots z_{n-2} \right|^2 \left( \overline{z_0} z_{n-1} - z_0 \overline{z_{n-1}} \right) \\
&= \left| z_1 \cdots z_{n-2} \right|^2 \, \Im \left( \overline{z_0} z_{n-1} \right) \\
&= \prod_{j=1}^{\frac{n-2}{2}} \left( m_{2j}^2 + 1 \right),
\end{align*}
where in the second line we used the fact that $(\alpha - \overline{\alpha}) = 2i \, \Im(\alpha)$ for any $\alpha \in \C$, and in the third line we used $\Im \left( z_0 \overline{z_{n-1}} \right) = -1$ and $m_{2j}^2 + 1 = |z_{2j-1}|^2 \, |z_{2j}|^2$ from \refeqn{m_from_z}. 

The odd case is similar, with an extra factor of $m_0^2 = \kappa_0 + 1 = |z_0|^2 $ appearing.
This proves the result with $\kappa_\infty$ set to $1$. Upon reversing the original dilation, we replace each $\kappa_j$ with $\frac{\kappa_j}{\kappa_\infty}$, completing the proof.

\small 

\bibliographystyle{amsplain}
\bibliography{refs.bib}

\end{document}